\documentclass[12pt]{article}

\usepackage[utf8]{inputenc}
\usepackage[T1]{fontenc}
\usepackage{verbatim}
\usepackage[margin=3cm]{geometry}
\usepackage[algo2e,ruled,vlined]{algorithm2e}
\usepackage{algorithm}
\usepackage{here}
\usepackage{setspace}
\usepackage[normalem]{ulem}
\usepackage{algpseudocode}
\usepackage{dsfont}
\usepackage[inline]{enumitem}
\usepackage{microtype}
\usepackage[affil-sl]{authblk}
\usepackage{graphicx}
\usepackage{nicematrix}
\usepackage{mathrsfs}
\usepackage{boldline,colortbl,float,import,pdflscape}

\usepackage{tikz}
\usetikzlibrary{tikzmark}
\usetikzlibrary{positioning}
\usetikzlibrary{graphs}
\usetikzlibrary{graphs.standard}
\usetikzlibrary{decorations.pathmorphing}
\usetikzlibrary{decorations.text}
\usetikzlibrary{patterns,shadings}
\usetikzlibrary{through,intersections,calc}
\usepackage{tikz-3dplot}

\usepackage{amssymb}
\usepackage{amsmath}
\usepackage{amsthm}
\usepackage{mathtools}
\usepackage{nicefrac}
\usepackage{amsfonts,latexsym,graphics}
\usepackage{color}
\usepackage[hidelinks]{hyperref}
\usepackage[hyphenbreaks]{breakurl}
\usepackage{xurl}
\hypersetup{breaklinks=true}
\usepackage{multirow}
\usepackage[english]{babel}
\usepackage{epsfig}
\usepackage{bm}
\usepackage{makeidx}
\usepackage{subcaption}
\usepackage{adjustbox}
\usepackage{diagbox}
\usepackage{float}

\usepackage{xspace}         

\theoremstyle{definition}
\newtheorem{definition}{Definition}

\theoremstyle{plain}
\newtheorem{proposition}[definition]{Proposition}
\newtheorem{lemma}[definition]{Lemma}
\newtheorem{theorem}[definition]{Theorem}

\newtheorem{corollary}[definition]{Corollary}

\newcommand*{\claimproofname}{Proof}

\def\B{{\mbox {\boldmath $B$}}}

\def\matrix0{{\mbox {\boldmath $O$}}}

\def\a{{\mbox{\boldmath $a$}}}
\def\e{{\mbox{\boldmath $e$}}}

\def\w{{\mbox{\boldmath $w$}}}
\def\x{{\mbox{\boldmath $x$}}}
\def\y{{\mbox{\boldmath $y$}}}
\def\z{{\mbox{\boldmath $z$}}}

\def\j{{\mbox{\boldmath $1$}}}
\def\vec0{\mbox{\bf 0}}

\def\vecx{{\mbox{\boldmath $x$}}}

\newcommand\tran{\mkern-2mu\raise1.25ex\hbox{$\scriptscriptstyle\top$}\mkern-3.5mu}

\def\Ba{{\mathcal{B}}}

\def\Sa{{\mathcal{S}}}

\def\Wa{{\mathcal{W}}}

\def\d{\mathop{\rm dist }\nolimits}

\def\diam{\mathop{\rm diam }\nolimits}
\renewcommand{\diam}{\mathrm{diam}}

\def\g{\mathop{\rm g }\nolimits}
\renewcommand{\d}{\mathrm{d}}

	{\end{list}}

\DeclareRobustCommand{\VAN}[3]{#2} 

\SetArgSty{}
\SetKwFor{For}{for}{do}{}
\SetKwInOut{Input}{Input}
\SetKwInOut{Output}{Output}

\definecolor{border}{RGB}{100,120,170}
\definecolor{inner}{RGB}{240,245,255}
\colorlet{inner}{gray!20}
\colorlet{inner2}{gray!80}
\colorlet{border}{black!80}

\makeatletter
\let\c@figure\c@table
\let\ftype@figure\ftype@table
\let\ext@figure\ext@table

\newcommand{\gauss}[2]{\genfrac{[}{]}{0pt}{}{#1}{#2}\@ifnextchar\bgroup{_}{}}
\makeatother

\mathtoolsset{centercolon}

\title{On the diameter and zero forcing number of some\\ graph classes in the Johnson, Grassmann and Hamming association scheme}
 
\author{Aida Abiad\thanks{\texttt{a.abiad.monge@tue.nl}, Department of Mathematics and Computer Science, Eindhoven University of Technology, The Netherlands\newline Department of Mathematics: Analysis, Logic and Discrete Mathematics, Ghent University, Belgium\newline Department of Mathematics and Data Science of Vrije Universiteit Brussel, Belgium} \qquad Robin Simoens\thanks{\texttt{Robin.Simoens@ugent.be},  Department of Mathematics: Analysis, Logic and Discrete Mathematics, Ghent University, Belgium} \qquad  Sjanne Zeijlemaker\thanks{\texttt{s.zeijlemaker@tue.nl},  Department of Mathematics and Computer Science, Eindhoven University of Technology, The Netherlands}}

\date{}

\begin{document}
\maketitle

\begin{abstract}
We determine the diameter of generalized Grassmann graphs and the zero forcing number of some generalized Johnson graphs, generalized Grassmann graphs and the Hamming graphs. Our work extends several previously known results.
\end{abstract}

\section{Introduction}

Graph classes in the Johnson, Grassmann and Hamming association scheme have received a considerable amount of attention over the last decades, and several graph properties and parameters have been investigated for these families. Examples are the work of Alspach, who showed that the Johnson graphs are Hamilton connected~\cite{alspach2013johnson}, the results of Bailey and Meagher, who studied the metric dimension of Grassmann graphs~\cite{meagher2012metric}, and the coloring results on generalized Kneser graphs by Balogh, Cherkashin and Kiselev~\cite{balogh2019coloring}. However, many parameters of these graphs are still unknown. In this paper, we focus on studying their diameter and zero forcing number.

Zero forcing is a propagation process on a graph where the vertices are initially partitioned into two sets of black and white vertices. A white vertex is colored black (\emph{forced}) if it is the unique white neighbor of a black vertex. The minimum number of initial black vertices needed to force all vertices of a graph~$G$ is called the \emph{zero forcing number}. Zero forcing was formally introduced in the AIM workshop~\cite{AIM} as an upper bound for the minimum rank, a connection which was also established by Alon~\cite{A}. The study of this parameter is motivated by the investigation of quantum networks,
influence in social networks and power dominating sets, see for example~\cite{burgarth2009local} and~\cite{dean2011power}.

In this paper, we calculate the girth and diameter of generalized Grassmann graphs, extending the known results for Kneser graphs~\cite{valencia2005diameter} and generalized Johnson graphs~\cite{agong2018girth} by  Valencia-Pabon and Vera, and Agong, Amarra, Caughman, Herman and Terada respectively. Recently, these results have also been complemented by Caughman, Herman and Terada \cite{CHT2023}, who studied the odd girth of generalized Johnson graphs. Moreover, we determine the zero forcing number of some generalized Johnson and generalized Grassmann graphs. As a corollary of our work, we obtain the known zero forcing results for Johnson graphs on~$2$-sets by Fallat, Meagher, Soltani and Yang.~\cite{ZplusJohnson} and for Kneser graphs by Bre{\u{s}}ar, Kos, and Torres~\cite{brevsar2019grundy}. Aazami posed a conjecture on the zero forcing of hypercubes \cite[Conjecture 2.4.4]{aazami2008hardness} which was proved by Alon \cite{A} and, independently, by several authors in a AIM workshop \cite{AIM}. Here we determine the zero forcing number of Hamming graphs, extending the mentioned known results.

\section{Preliminaries}
\label{sec:prelim}

Let~$G=(V,E)$ be a graph with vertex set~$V$ and edge set~$E$. Throughout this paper, we consider simple graphs, i.e.\ undirected, loopless graphs without multiple edges. Adjacency of vertices~$v$ and~$w$ will be denoted by~$v\sim w$ and the open and closed
neighborhood of a vertex~$v$ by~$G(v)$ and~$G[v]$ respectively. The \textit{induced subgraph}~$G[S]$ on a subset~$S\subseteq V$ is the graph with vertices~$S$ and edges~$\{e\in E\mid e\subseteq S\}$. The \textit{Cartesian product} of two graphs~$G$ and~$H$, denoted by~$G\square H$, is the graph on~$V(G)\times V(H)$ such that~$(u, v)$ is adjacent to~$(u',v')$ if and only if~$u = v$ and~$u'$ is adjacent to~$v'$ in~$H$, or~$u' = v'$ and~$u$ is adjacent to~$v$ in~$G$. If graphs $G$ and $H$ are isomorphic, we denote this by $G\cong H$. 

The \textit{distance} $\d(v,w)$ between two vertices~$v$ and~$w$ is the length of a shortest walk from~$v$ to~$w$. The \textit{diameter}~$\diam{(G)}$ is the largest possible distance between two vertices of~$G$. The \textit{girth}~$\g(G)$ is the length of a shortest cycle.

Let~$n$,~$k$ be positive integers with~$n\ge k$ and let~$S\subseteq\{0,1,\dots,k-1\}$. The \emph{generalized Johnson graph}~$J_S(n,k)$ has as vertices the~$k$-subsets of $\{1,2,\dots,n\}=:[n]$. Two vertices are adjacent if their intersection size is in~$S$. Well-known subfamilies include the \emph{Kneser graph} $K(n,k)=J_{\{0\}}(n,k)$ and the \emph{Johnson graph} $J(n,k)=J_{\{k-1\}}(n,k)$. Note that if~$n\le 2k$, the Kneser graph is either empty or a disjoint set of edges, hence we assume that~$n\ge 2k+1$ for these graphs.

The generalized Johnson graphs have the following~$q$-analogue. Let~$\mathbb{F}_q$ be the finite field of order~$q$, where $q$ is a prime power. A subspace of $\mathbb{F}_q$ of dimension~$k$ is called a \emph{$k$-subspace} for short. If we let~$S\subseteq\{0,1,\dots,k-1\}$, the \emph{generalized Grassmann graph~$J_{q,S}(n,k)$} has as vertices the~$k$-subspaces of~$\mathbb{F}_q^n$ and vertices are adjacent if the dimension of their intersection is in~$S$. Particular examples include the \emph{q-Kneser graph}~\(K_q(n,k)=J_{q,\{0\}}(n,k)\) and the \emph{Grassmann graph}~\(J_q(n,k)=J_{q,\{k-1\}}(n,k)\). The number of vertices of $J_{q,S}(n,k)$ is the Gaussian binomial coefficient
\[\gauss{n}{k}{q} \coloneqq \frac{(1-q^n)(1-q^{n-1})\dots (1-q^{n-r+1})}{(1-q)(1-q^2)\dots(1-q^k)}.\]

Note that we may assume~$n\ge 2k$ since~$J_S(n,k)\cong J_{\{s+n-2k\mid s\in S\}}(n,n-k)$, and similarly,~\(J_{q,S}(n,k)\cong J_{q,\{s+n-2k\mid s\in S\}}(n,n-k)\).

In Johnson graphs, vertices are associated with sets and adjacency is determined by the cardinality of their intersection. If we replace sets by ordered tuples, we obtain the well-known Hamming graphs~$H(n,q)$, where~$n$ denotes the tuple length and the entries range from~$0$ to~$q-1$. Tuples are connected by an edge whenever they coincide in~$n-1$ coordinates, i.e., when their Hamming distance is $1$. Alternatively, one may view~$H(n,q)$ as the  Cartesian product of~$n$ copies of~$K_q$. 

\textit{Zero forcing} is a graph coloring problem in which an initial subset of vertices is
colored black, while the others are colored white. A black vertex~$v$ colors a white neighbor~$w$ black if~$G(v)\setminus w$ is entirely black, i.e.\ if~$w$ is the unique white neighbor of~$v$.  A black vertex which initiates a color change is called a \textit{pivot} and we refer to the initial set of black vertices as the \textit{leader set}. If all white vertices can be colored black under the given coloring rule, we call the leader set a \textit{zero forcing set}. The \textit{zero forcing number} of a graph~$G=(V,E)$, denoted by~$Z(G)$, is the cardinality of a smallest subset of~$V$ which is a zero forcing set.

By modifying the color changing rule or adding restrictions to the zero forcing set, several interesting variations on zero forcing can be obtained. Below, we list the ones that are relevant for the results in this article. A more complete overview of these parameters and their mutual relations can be found in~\cite{barioli2013parameters,IEPbook}.

The \textit{connected zero forcing number}~$Z_c(G)$ is the cardinality of the smallest zero forcing set~$S$ such that~$G[S]$ is connected~\cite{BRIMKOV201731}. The \textit{total zero forcing number}~$Z_t(G)$ is the cardinality of the smallest zero forcing set~$S$ such that~$G[S]$ has no isolated vertices~\cite{totalZF}. Both~$Z_c(G)$ and~$Z_t(G)$ upper bound the zero forcing number. A single vertex can only be a zero forcing set if~$G$ is a path, hence if~$G \neq P_{|V|}$, we have~$2 \le Z(G) \leq Z_t(G) \leq Z_c(G)$.

For a graph~$G$, let~$\Sa^{\mathbb{F}}(G)$ be the set of symmetric matrices~$M$ over the field~$\mathbb{F}$ with~$m_{ij} = 0$ if and only if vertices~$i\neq j$ are not adjacent and arbitrary entries on the diagonal. The \textit{maximum nullity} of~$G$ over~$\mathbb{F}$, denoted by~$M^{\mathbb{F}}$, is the largest multiplicity of eigenvalue zero for any matrix in~$\Sa^{\mathbb{F}}(G)$. It was shown in~\cite{AIM} and~\cite{A} that the maximum nullity of a graph over any field lower bounds the zero forcing number.

\begin{lemma}[{\cite{AIM}, Proposition 2.4 and \cite{A}, Theorem 2.1}]
    For any graph~$G$ and field~$\mathbb{F}$, $M^{\mathbb{F}}(G) \le Z(G)$.
    \label{lem:maxnul}
\end{lemma}

Besides maximum nullity, zero forcing is closely related to other graph parameters, such as metric dimension~\cite{eroh2017comparison}, power domination~\cite{dean2011power} and Grundy domination~\cite{BRESAR2017}. A vertex~$v$ \textit{dominates}~$w$ if~$w\in G[v]$. A sequence~$(v_1,v_2,\dots,v_\ell)$ of vertices in~$V$ is called a \textit{dominating sequence} for~$G$ if each~$w\in V$ is dominated by some~$v_i$. If additionally for each~$i$ it holds that 
\[G[v_i]\setminus\bigcup_{j=1}^{i-1}G[v_j]\neq \emptyset,\]
we call~$(v_1,v_2,\dots,v_\ell)$ a \textit{Grundy dominating sequence} and if for each $i$ we have
\[G(v_i)\setminus\bigcup_{j=1}^{i-1}G[v_j]\neq \emptyset,\]
the sequence is called \textit{Z-Grundy dominating}. The vertex~$v_i$ is said to \textit{footprint} the vertices in~$G[v_i]\setminus\cup_{j=1}^{i-1}G[v_j]$ or $G(v_i)\setminus\cup_{j=1}^{i-1}G[v_j]$, respectively. In other words, a vertex~$w$ is footprinted by~$v_i$ if~$i$ is the smallest integer such that~$v_i$ dominates~$w$. Note that every vertex has a unique footprinter. The maximum length of a Grundy or Z-Grundy dominating sequence, denoted by~$\gamma_{gr}(G)$ or $\gamma_{gr}^Z(G)$ respectively, is referred to as the \textit{(Z-)Grundy domination number} of~$G$. 

Grundy domination is known to give a lower bound on the zero forcing number, a fact that we will use in some of our proofs. 

\begin{lemma}[{\cite[Corollary 2.3]{BRESAR2017}}]\label{lem:grundy}
Let~$G$ be a graph without isolated vertices, then
\begin{itemize}
    \item $Z(G) \ge |V| - \gamma_{gr}(G)$;
    \item $Z(G) = |V| - \gamma^Z_{gr}(G)$.
\end{itemize}
\end{lemma} 

For an overview of Grundy domination variants and their relation to zero forcing, we refer to~\cite{BRESAR2017}.

\section{Diameter and girth of generalized Grassmann graphs}
\label{sec:diameter}

We extend the results of \cite{agong2018girth} on the girth and diameter of Johnson graphs to generalized Grassmann graphs. Our main tool will be the following lemma.

\begin{lemma}[\cite{bose1966characterization}]\label{lemma:sparse}
    Let \(n\geq k+m\). Given at most \(q^{n-k-m+1}\) \(k\)-spaces in \(\mathbb{F}_q^n\), we can always find an \(m\)-space that intersects them trivially.
\end{lemma}
\begin{proof}
    Let \(a\leq q^{n-k-m+1}\) be the number of given \(k\)-spaces in \(\mathbb{F}_q^n\). We find a suitable \(m\)-space by constructing a basis \(\{\vecx_1,\vecx_2,\dots,\vecx_m\}\) for it. For the first basis vector, \(\vecx_1\), there are at least
    \[(q^n-1)-a(q^k-1)\]
    choices: the total number of nonzero vectors, minus those that lie in one (and possibly more) of the \(a\) given \(k\)-spaces.
    For \(\vecx_2\), there are at least
    \[(q^n-1)-(q-1)-a(q^{k+1}-q)\]
    choices: the total number of nonzero vectors, minus those that span the same \(1\)-space as \(\vecx_1\), minus those that lie in the span of \(\vecx_1\) and one of the given \(k\)-spaces, but are no multiple of \(\vecx_1\) (because we eliminated those already).
    Continuing in this way, we find a decreasing number of available vectors, ending with at least
    \[(q^n-1)-(q^{m-1}-1)-a(q^{k+m-1}-q^{m-1})\]
    choices for \(\vecx_m\). Since \(a\leq q^{n-k-m+1}\), this number is at least \(q^n-q^{m-1}-q^n+q^{n-k}=q^{n-k}-q^{m-1}\), which is strictly positive since \(k+m\leq n\). We conclude that such a space exists.
\end{proof}

In the statement below, we call $A$ a \textit{proper subset} of $B$, denoted by $A\subset B$, if $\emptyset\neq A \neq B$. 

\begin{theorem}\label{thm:diamgrassmann}
    Let \(n\geq2k\) and suppose that \(S \subset\{0,1,\dots,k-1\}\) with \(s\coloneqq \min(S)\). The generalized Grassmann graph \(J_{q,S}(n,k)\) has diameter 
    \[\diam\left(J_{q,S}(n,k)\right)=\left\{\begin{array}{ll}2&\text{if }s=0\\\left\lceil\frac{k}{k-s}\right\rceil&\text{if }s\neq0.\end{array}\right.\]
    \vspace{-3mm}
\end{theorem}
\begin{proof}
    Let \(v\) and \(w\) be two arbitrary vertices with intersection dimension \(t\). We will show that
    \[\d(v,w)=\left\{\begin{array}{ll}\left\lceil\frac{k-t}{k-s}\right\rceil&\text{if }t<s\\1\text{ or }2&\text{if }t\geq s.\end{array}\right.\]

\begin{description}
\item[Case 1: \(t<s\).] Then in particular, \(v\not\sim w\). We must prove that \(\d(v,w)=d'\), where \(d'\coloneqq\left\lceil\frac{k-t}{k-s}\right\rceil\). We distinguish two cases, similarly to the proof of \cite[Lemma~3.2]{agong2018girth}.
\begin{description}
\item[Case 1(a): \(k+t\geq2s\).]  In this case, \(k-s<k-t\leq2(k-s)\), so \(d'=2\). It suffices to construct a \(k\)-space that intersects both \(v\) and \(w\) in an \(s\)-space. Choose an \(s\)-space \(\pi\) through \(v\cap w\) in \(v\) and an \(s\)-space \(\tau\) through \(v\cap w\) in \(w\). These two span a \((2s-t)\)-space, which we can extend to a \(k\)-space that intersects \(v\) and \(w\) in just this space, by applying Lemma~\ref{lemma:sparse} to its residue. Therefore, \(\d(v,w)=2=d'\).
\item[Case 1(b): \(k+t<2s\).] We first show that \(\d(v,w)\leq d'\) by constructing a walk of length \(d'\) from \(v\) to \(w\). Choose a basis \(\{\x_1,\x_2,\dots,\x_t\}\) of \(v\cap w\) and expand it to a basis \(\{\x_1,\x_2,\dots,\x_t,\y_1,\y_2,\dots,\y_{k-t}\}\) of \(v\) and to a basis \(\{\x_1,\x_2,\dots,\x_t,\z_1,\z_2,\dots,\z_{k-t}\}\) of \(w\). Define
    \[u_i\coloneqq\langle \x_1,\x_2,\dots,\x_t,\y_{i(k-s)+1},\y_{i(k-s)+2},\dots,\y_{k-t},\z_1,\z_2,\dots,\z_{i(k-s)}\rangle\]
    for \(i\in\{1,2,\dots,d'-2\}\). Then \((v,u_1,u_2,\dots,u_{d'-2})\) is a walk of length \(d'-2\). The dimension $t'$ of the intersection of \(u_{d'-2}\) and \(w\) is at least \(2s-k\), hence these vertices satisfy the condition $k+t'\ge 2s$ from Case 1a. This means that $d(u_{d'-2},w) = 2$, so we can extend the walk to \(w\) by only two more steps. This results in a walk of length \(d'\) between \(v\) and \(w\).
    
    \begin{figure}[H]
        \centering
        \begin{tikzpicture}[scale=2, remember picture]
    \draw[thick, border,fill=inner,rotate around={-50:(-.08,0)}] (-.08,1.2) ellipse (.5 and 1.4) {};
    \draw[thick, border,fill=inner!50,rotate around={-30:(-.04,0.02)}]
    (-.15,1.17) ellipse (.8 and 1.4) {};
    \draw[thick, border,fill=inner] (0,1.2) ellipse (.9 and 1.4) {};
    \draw[thick, border,fill=inner!50,rotate around={30:(.04,.02)}]
    (.15,1.17) ellipse (.8 and 1.4) {};
    \draw[thick, border,fill=inner,rotate around={50:(.08,0)}] (.08,1.2) ellipse (.5 and 1.4) {};
    \draw[border,rotate around={-50:(-.08,0)}] (-.08,1.2) ellipse (.5 and 1.4) {};
    \draw[border,rotate around={-30:(-.04,0.02)}]
    (-.15,1.17) ellipse (.8 and 1.4) {};
    \draw[border] (0,1.2) ellipse (.9 and 1.4) {};
    \draw[border,rotate around={30:(.04,.02)}]
    (.15,1.17) ellipse (.8 and 1.4) {};
    \draw[border,rotate around={50:(.08,0)}] (.08,1.2) ellipse (.5 and 1.4) {};
    \path[every node/.append style={circle, fill=black, minimum size=4pt, label distance=5pt, inner sep=0pt}]
    (-.1,.1) node {}
    (0.02,.3) node[label={[label distance=12pt]270:\small\(\dots\)}] {}
    (.1,.12) node {};
    \path (-2.1,1.6) node[label=\(v\)] {}
    (2.1,1.6) node[label=\(w\)] {}
    (-1.4,2.2) node[label=\(u_1\)] {}
    (1.4,2.2) node[label=\(u_{d'-2}\)] {}
    (0.25,2.4) node[label={[label distance=0pt]180:\(\dots\)}] {}
    (.04,-0.05) node[label={[label distance=3pt]270:\(\x_1,\x_2,...,\x_t\)}] {};
    \path[every node/.append style={circle, fill=black, minimum size=4pt, label distance=2pt, inner sep=0pt}]
    (-1.65,1.5) node[label={[label distance=0pt]150:\(\y_1\)}] {}
    (-1.45,.9) node[label={[label distance=-5pt]210:\(\y_{k-s}\)}] {}
    (1.65,1.5) node[label={[label distance=2pt]0:\(\z_{(d'-2)(k-s)+1}\)}] {}
    (1.6,1.3) node {}
    (1.4,.8) node[label={[label distance=0pt]-30:\(\z_{k-t}\)}] {}
    (-1.12,1.2) node[label={[label distance=-7pt]60:\(\y_{k-s+1}\)}] {}
    (-.95,.6) node[label={[label distance=-10pt]270:\(\y_{2(k-s)}\)}] {}
    (1.1,1.25) node {}
    (.95,.6) node[label={[label distance=-10pt]300:\(\z_{(d'-2)(k-s)}\)}] {}
    (-.3,.75) node {}
    (-.33,.6) node[label={[label distance=-2pt,rotate=-5]270:\(\vdots\)}] {}
    (-.37,.2) node {}
    (.3,.65) node[label={[label distance=-2pt]90:\(\z_1\)}] {}
    (.37,.2) node[label={[label distance=5pt,rotate=10]90:\(\vdots\)}][label={[label distance=1pt]-1:\(\z_{k-s}\)}] {};
    \path
    (-.7,.6) node[label=\(\ddots\)] {}
    (.7,.6) node[label=\reflectbox{\(\ddots\)}] {}
    (-1.57,1.28) node[label={[label distance=-14pt,rotate=20]270:\(\vdots\)}] {}
    (-1.1,1.1) node[label={[label distance=-7pt,rotate=15]270:\(\vdots\)}] {}
    (1.08,1.1) node[label={[label distance=-8pt,rotate=-15]270:\(\vdots\)}] {}
    (1.58,1.3) node[label={[label distance=-3pt,rotate=-20]270:\(\vdots\)}] {};
\end{tikzpicture}
        \caption{The walk \((v,u_1,u_2,\dots,u_{d'-2})\), drawn projectively.}
    \end{figure}
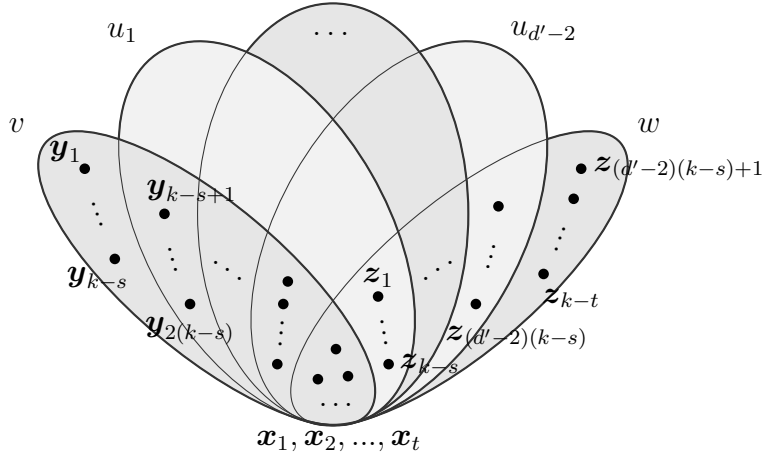
    
    In order to show the converse inequality \(\d(v,w)\geq d'\), it suffices to prove that \(\d(v,w)\geq\frac{k-t}{k-s}\), since \(\d(v,w)\) is an integer. We do this by induction on \(\d(v,w)\). If \(\d(v,w)=1\), then \(\d(v,w)\geq\frac{k-t}{k-s}\) simply because \(s\leq t\). For the induction step, consider a vertex \(u\) with \(\d(u,v)=\d(v,w)-1\) and \(u\sim w\). The induction hypothesis implies that \(\d(u,v)\geq\frac{k-\dim(u\cap v)}{k-s}\), or after rewriting, \((s-k)\cdot\d(v,w)-s+2k\leq\dim(u\cap v)\). Thus,
    \begin{align*}
        (s-k)\cdot\d(v,w)+2k&\leq\dim(u\cap v)+s\\
        &\leq\dim(u\cap v)+\dim(u\cap w)\\
        &=\dim(u\cap v\cap w)+\dim(u\cap\langle v,w\rangle)\\
        &\leq\dim(v\cap w)+\dim(u)\\
        &=t+k,
    \end{align*}
    where we used that \(u\sim w\) in the second step, and applied the Grassmann identity in the third step. We conclude that \(\d(v,w)\geq\frac{k-t}{k-s}\).
\end{description}
\item[Case 2: \(t\geq s\).] If \(t\geq s\), we can choose an \(s\)-space \(\pi\) in the intersection of \(v\) and \(w\) and construct a \(k\)-space \(u\) that intersects \(v\) and \(w\) in \(\pi\) using Lemma~\ref{lemma:sparse}, like we did in Case 1. This construction provides us with a walk of length \(2\), so the distance \(\d(v,w)\) is at most \(2\).
\end{description}

If \(s=0\), then we always have \(t\geq s\) and the result follows from Case 2. If \(s\neq0\), then \(\left\lceil\frac{k}{k-s}\right\rceil\) is greater than \(2\) and the result follows, because \(\left\lceil\frac{k-t}{k-s}\right\rceil\) is maximal for \(t=0\).
\end{proof}

Allowing \(n\) to be smaller than \(2k\) leads to the following, more general statement. Note that the condition \(n\geq 2k-\max(S)\) below is just  there to ensure connectivity.

\begin{corollary}\label{cor:diamgrassmann}
    Let \(S\) be a proper subset of \(\{0,1,\dots,k-1\}\) with \(s\coloneqq \min(S)\) such that \(n\geq2k-\max(S)\). The generalized Grassmann graph \(J_{q,S}(n,k)\) has diameter 
    \[\diam\left(J_{q,S}(n,k)\right)=\left\{\begin{array}{ll}2&\text{if }s\in\{0,2k-n\}\\\left\lceil\frac{\min(k,n-k)}{k-s}\right\rceil&\text{if }s\notin\{0,2k-n\}.\end{array}\right.\]
    \vspace{-3mm}
\end{corollary}
\begin{proof}
    We distinguish two cases.
    \begin{description}
\item[Case 1: \(n\geq2k\).] We are done by Theorem \ref{thm:diamgrassmann}.
\item[Case 2: \(n<2k\).] By definition, \(J_{q,S}(n,k)\cong J_{q,\{s+n-2k\mid s\in S\}}(n,n-k)\). Applying the first case on the latter graph (since \(2(n-k)<n\) and \(S\subseteq\{0,1,\dots,n-k-1\}\)), results in the diameter being \(\left\lfloor\frac{n-k}{n-k-(s+n-2k)}\right\rfloor=\left\lfloor\frac{n-k}{k-s}\right\rfloor\) if \(s+n-2k\neq0\) and \(2\) if \(s+n-2k=0\).\qedhere
\end{description}
\end{proof}

We end this section with a short proof of the girth of generalized Grassmann graphs.

\begin{proposition}\label{thm:girthgrassmann}
    Every generalized Grassmann graph \(J_{q,S}(n,k)\) with \(S\neq\emptyset\) has girth \(3\).
\end{proposition}
\begin{proof}
    Let \(J_{q,S}(n,k)\) be a nontrivial Grassmann graph and let \(s\in S\). Recall that we may assume that \(n\geq2k\) without loss of generality. Choose two \(k\)-spaces \(v\) and \(w\) that intersect in an \(s\)-space \(\pi\). Since \(2\leq q\leq q^{n-2k+1}\leq q^{(n-s)-2(k-s)+1}\), we can apply Lemma~\ref{lemma:sparse} to the residual projective space of \(\pi\) to find a third \(k\)-space \(u\) that intersects \(v\) and \(w\) in \(\pi\). Then \(u\), \(v\) and \(w\) are mutually adjacent, i.e.\ there is a triangle in the graph. We conclude that the girth must be \(3\).
\end{proof}

\section{Zero forcing number}

Determining the zero forcing number is known to be an NP-complete problem. The proof of this result is often wrongly attributed to Aazami \cite{aazami2008hardness, A2010} (see, e.g. \cite{BRIMKOV2019,BRIMKOV201731, IEPbook,TD2015}), who showed the NP-completeness for weighted zero forcing. However, this is incorrect, since the unweighted version does not follow from the weighted $0,1$ version that Aazami established in \cite[Theorem 2.3.1]{aazami2008hardness}. The NP-completeness of the unweighted variant is proved in~\cite{YANG2013100}, where it is shown that fast-mixed searching is NP-complete, and it is known that this problem is equivalent to zero forcing~\cite{fallat2016complexity}.
Despite its NP-completeness, the exact zero forcing number is known for certain graph classes~\cite{AIM, brevsar2016dominating,BRIMKOV201731}. In this section, we determine the zero forcing number of several families of generalized Johnson graphs, generalized Grassmann graphs and Hamming graphs.

\subsection{Generalized Johnson graphs}\label{sec:genJohnson}

First, we investigate the zero forcing number of generalized Johnson graphs, generalizing known results on Johnson and Kneser graphs from~\cite{ZplusJohnson} and~\cite{brevsar2019grundy}. A natural way to generalize the Johnson graphs, is to extend their connection set~$\{k-1\}$ to~$S=\{s,s+1,\dots,k-1\}$. We determine the zero forcing number of these graphs by studying the more general case where~$\min(S)=s$. Note that if~$n<2k-s$, every two~$k$-sets intersect in strictly more than~$s$ elements, so not all options in~$S$ can actually occur. We therefore assume that~$n\ge 2k-s$. To prove a lower bound on the zero forcing number, we will make use of the following variant of a theorem of Bollob\'as.
\begin{lemma}[{\cite[Theorem 1 + Remark 3.2]{furedi84}}]
\label{lem:bollobas}
Let~$X_1,X_2,\dots, X_m$ be~$r$-element sets, let $Y_1,Y_2,\dots,Y_m$ be~$s$-element sets and let $t\ge 0$ such that
\begin{description}
    \item[$(i)$] $|X_i\cap Y_i| \le t$ for~$i\in [m]$;
    \item[$(ii)$] $|X_i\cap Y_j| \ge t+1$ for~$1\le i<j\le m$.
\end{description}
Then~$m\le \binom{r+s-2t}{r-t}$. 
\end{lemma}

\begin{theorem}
Let~$S\subseteq \{0,1,\dots,k-1\}$ with~$s \coloneqq \min(S)$ and let~$n\ge 2k-s$. Then
\[Z(J_S(n,k)) \le Z_t(J_S(n,k)) \le Z_c(J_S(n,k)) \le \binom{n}{k} - \binom{n-2(k-s)}{s}.\]
If~$S = \{s,s+1,\dots,k-1\}$, equality holds throughout.
\label{th:genJohnson}
\end{theorem}

\begin{proof}
Let~$\Wa \coloneqq \{v\in V \mid [k-s] \subset v,\ k-s+1,k-s+2,\dots,2(k-s) \notin v\}$ and~$\Ba\coloneqq V \setminus \Wa$. Note that~$\Ba$ is well-defined, since~$n\ge 2(k-s)$. This set has size~$\binom{n}{k} - \binom{n-2(k-s)}{s}$ and we will show that it is a zero forcing set.

Consider a white vertex~$v\coloneqq\{1,2,\dots,k-s,v_1,v_2,\dots,v_{s}\}\in \Wa$, with~$v_i \notin \{k-s+1,k-s+2,\dots,2(k-s)\}$. It is adjacent to~$w\coloneqq\{k-s+1,k-s+2,\dots,2(k-s),v_1,v_2,\dots,v_{s}\}\in \Ba$. Every other white vertex has at most~$s-1$ elements in common with~$w$, because their intersection must be a subset of~$\{v_1,v_2,\dots,v_s\}$. Therefore,~$v$ is the unique white neighbor of~$w$ and is forced black. This holds for any~$v\in \Wa$, so~$\Ba$ is a zero forcing set.

It is also a connected zero forcing set, since we can construct a black path between any two black vertices as follows.
Let~$v,v'\in \mathcal{B}$. For any choice of~$S$, vertices are adjacent if they differ in~$k-s$ elements. Using this property, we can replace one element~$a\in v$ at the time with a new element~$b$ to eventually arrive at~$v'$: choose a~$(2k-s)$-set~$X$ containing~$v\cup\{a,b\}$ and a subset~$Y\subset v$ such that~$|Y| = k-s$,~$a\in Y$. Then there exists a path~$v \rightarrow X\setminus Y \rightarrow (v\setminus a) \cup b$, where at each step, we replace exactly~$k-s$ elements. Moreover, by choosing~$Y$ appropriately we can ensure that all path vertices are black. Therefore, the leader set is connected.

Now assume that~$S=\{s,s+1,\dots,k-1\}$. To prove that the upper bound is tight, we use the Grundy domination number of~$J_S(n,k)$. Consider a Grundy dominating sequence~$v_1,v_2,\dots, v_m$ of maximum length and pair each~$v_i$ with a vertex~$w_i$ that is footprinted by it. Define the sets~$X_i = v_i$ and~$Y_i = [n] \setminus w_i$. Then~$|X_i\cap Y_i|\le k-s$. Moreover, if $j>i$, we know that~$|v_i\cap w_j| \le s - 1$, otherwise~$w_j$ is either~$v_i$ or one of its neighbors, so it would have been dominated by~$v_i$ before~$v_j$. Hence~$|X_i\cap Y_j| \ge k-s + 1$ for~$i<j$. Lemma~\ref{lem:bollobas} implies that~$\gamma_{gr}(J_S(n,k))\le \binom{n-2(k-s)}{s}$, hence
\[Z(J_S(n,k)) \ge \binom{n}{k} - \gamma_{gr}(J_S(n,k))\ge \binom{n}{k} - \binom{n-2(k-s)}{s}, \]
where the first inequality follows from Lemma~\ref{lem:grundy}.
\end{proof}

If~$s=k-1$, we obtain the zero forcing number of the classical Johnson graphs.

\begin{corollary} For $n\ge 2k$,
$Z(J(n,k))=Z_t(J(n,k)) = Z_c(J(n,k)) = \binom{n}{k} - \binom{n-2}{k-1}.$
\end{corollary}

The Kneser graphs can be generalized in a similar way by extending the intersection set to~$\{0,1,\dots,t\}$ for some integer~$t \le k-1$. This generalization also appears in~\cite{balogh2019coloring}, where the coloring number of generalized Kneser graphs is studied, and~\cite{CIOABA2018219}, which shows that some of these graphs are not determined by their spectrum. As usual for Kneser graphs, we require~$n\ge 2k+1$.

\begin{theorem}\label{th:genKneserLower}
Let~$n\ge 2k+1$ and $S=\{0,1,\dots,t\}$ for some $t \in \{0,1,\dots,k-1\}$. Then
\[Z(J_S(n,k)) \ge \binom{n}{k} - \binom{2k-2t}{k-t}.\]
\end{theorem}
\begin{proof}
Consider a Z-Grundy dominating sequence~$v_1,v_2,\dots, v_t$ of maximum length and pair each~$v_i$ with a vertex~$w_i$ that is footprinted by it. Then~ $|v_i\cap w_i|\le t$ and if $j>i$,~$|v_i\cap w_j| \ge t + 1$. Lemma~\ref{lem:bollobas} implies that~$\gamma^Z_{gr}(J_S(n,k))\le \binom{2k-2t}{k-t}$, hence
    \[Z(J_S(n,k)) \ge \binom{n}{k} - \gamma^Z_{gr}(J_S(n,k))\ge \binom{n}{k} - \binom{2k-2t}{k-t}.\qedhere\] 
\end{proof}

Note that in the proof of Theorem \ref{th:genKneserLower} we need a Z-Grundy dominating sequence specifically. In a normal Grundy dominating sequence, a vertex could footprint itself, in which case the intersection $|v_i\cap w_i|$ would equal~$k > t$. 

\begin{theorem}
Let~$S\subseteq\{0,1,\dots,k-2\}$ and $n\ge \max(3k-2t+1,2k+1)$, where $t\coloneqq\max(S)$. Then
\[Z(J_S(n,k)) \le Z_t(J_S(n,k)) \le Z_c(J_S(n,k)) \le \binom{n}{k} - \binom{2k-2t}{k-t}.\]
If~$S=\{0,1,\dots,t\}$, equality holds throughout.
\label{th:genKneserUpper}
\end{theorem}

\begin{proof}
Define~$T \coloneqq \{n-2k+2t+2,n-2k+2t+3,\dots, n\}$ and~$\Wa \coloneqq\Wa_1\cup\Wa_2$, where
\begin{align*}
 \Wa_1 &= \{v \in V\mid |v\cap T| = k-t,\ [t]\subset v\},\\
 \Wa_2 &=\{v\in V \mid v\subset [2k-t],\ [t+1] \subset v\},
\end{align*}
and let $\Ba \coloneqq V \setminus \Wa$.
Since~$n \ge 2k-t$, we know that~$t+1 \notin T$. Then~$\Wa$ contains
\[\binom{2k-2t-1}{k-t-1} + \binom{2k-2t - 1}{k-t} = \binom{2k-2t}{k-t}\] 
distinct vertices,
so~$\Ba$ has the desired cardinality. We will show that it is a connected zero forcing set by first forcing the vertices in~$\Wa_1$ and then those in~$\Wa_2$.

Let~$v\in \Wa_1$ and let~$w$ be the vertex~$(T\setminus v)\cup [t + 1]$. Note that~$w\sim v$, since their intersection is~$[t]$. For any other~$v'\in \Wa_1$, we have~$|w\cap v'| = |[t]\cup( (T\setminus v)\cap v')|\ge t+1$, so~$w\not\sim v'$. Moreover,~$|w\cap v''| \ge |[t+1]| = t+1$ for any~$v''\in \Wa_2$, so~$w$ has no white neighbors besides~$v$. If~$w$ is black, it will therefore force~$v$. This is the case if~$(T\setminus [2k-t])\not\subset v$, because then we cannot have~$w\in \Wa_2$. If this should hold for any choice of~$v$, we need~$|T\setminus [2k-t]| = n-2k+t > k-t$, which is satisfied by assumption, as~$n>3k-2t$. Therefore,~$\Wa_1$ can be forced entirely.

Next, consider a vertex~$v\in \Wa_2$ and let~$w=([2k-t]\setminus v)\cup [t] \in \Ba$. Then~$w\sim v$, but~$w$ is not adjacent to any other~$v'\in \Wa_2$, because~$|w\cap v'| = |[t] \cup (([2k-t]\setminus v)\cap v')|\ge t+1$. This means that~$v$ is the unique white neighbor of~$w$, so it will be colored black. 

The above construction also holds for the connected zero forcing number. The proof is analogous to that of Theorem~\ref{th:genJohnson}. If~$S = \{0,1,\dots,t\}$, it follows from Theorem~\ref{th:genKneserLower} that the upper bound is tight for the (connected) zero forcing number of~$J_S(n,k)$.
\end{proof}

The construction in Theorem~\ref{th:genKneserUpper} no longer works if~$n\le 3k-2t$, as some pivots of~$\Wa_1$ will be contained in the white set~$\Wa_2$, which is forced last. However, for the extremal case~$n=3k-2t$, we can construct a different zero forcing set of the same cardinality.

\begin{theorem}
Let~$S\subseteq\{0,1,\dots,k-2\}$ with~$t\coloneqq\max(S)$, and let~$n=3k-2t \ge 2k+1$. Then
\[Z(J_S(n,k)) \le Z_t(J_S(n,k)) \le Z_c(J_S(n,k)) \le \binom{n}{k} - \binom{2k-2t}{k-t}.\]
If~$S=\{0,1,\dots,t\}$, equality holds throughout.
\label{th:genKneserEdgeCase}
\end{theorem}
\begin{proof}
Consider the sets~$T$, $\Wa$ and $\Ba$ from Theorem~\ref{th:genKneserUpper}. The zero forcing process used  in the construction in Theorem~\ref{th:genKneserUpper} is no longer applicable when~$n=3k-2t$ because the white vertex~$v\coloneqq[t+1]\cup (T\cap [2k-t]) \in \Wa_2$ now acts as a pivot for~$[t]\cup (T\setminus [2k-t])\in \Wa_1$. However, we want to force~$\Wa_1$ before~$\Wa_2$. Therefore, the zero forcing process will no longer color the entire graph. We propose the following change.

Let~$x\in T\setminus [2k-t]$,~$y\in T\cap [2k-t]$. Define~$X = (T\cap [2k-t])\setminus \{y\}$. We add the previously white vertex~$v=[t+1]\cup (T\cap [2k-t])$ to the leader set $\Ba$, and instead color~$v' = [t]\cup \{t+2\} \cup X\cup \{x\}$ white (note that~$|v'| = t + 1 + (k-t-2) + 1 = k$, so this is indeed a~$k$-set). We will show that this gives a zero forcing set. 

Using the same pivots as in the proof of Theorem~\ref{th:genKneserUpper}, all vertices of~$\Wa_1$ can be colored black, except those containing~$X \cup \{x\}$. We will force those later and focus on~$\Wa_2$ first. 

Consider the pivots from Theorem~\ref{th:genKneserUpper} corresponding to~$\Wa_2$. A vertex cannot be forced if its pivot is adjacent to~$v'$ or a white vertex of~$\Wa_1$. It meets at least one of these conditions whenever it contains~$X$. Let~$v''$ be such a vertex. Its pivot contains~$y$, as~$v''\neq v$. Replace~$y$ by~$x$, then the resulting vertex has no other white neighbors in~$\Wa_2$ and is not adjacent to~$v'$ or any white vertex from~$\Wa_1$ (both contain~$x$). Moreover, it is black, so it can force~$v''$. Hence all of~$\Wa_2$ can be colored black. 

Now consider the vertex~$[t]\cup (T\setminus (\{i\}\cup [2k-t]))\cup \{y\}$. This vertex from~$\Wa_1$ has been forced already and is not adjacent to any white vertices in~$\Wa_1$, which must contain at least one element from~$(T\setminus (\{i\}\cup [2k-t]))\cup\{y\}$. It can therefore force~$v'$. The remainder of~$\Wa_1$ can then be forced by their usual pivots.

Once again, tightness of this construction for~$S=\{0,1,\dots,t\}$ follows from Theorem~\ref{th:genKneserLower}. The connectivity argument is analogous to the one in the proof of Theorem~\ref{th:genJohnson}.
\end{proof}

Suppose~$\Ba$ is a zero forcing set for~$J_{\{0,1,\dots,t\}}(n,k)$ and~$\Wa=V\setminus\Ba$. Then for any~$n'> n$,~$V \setminus \Wa$ is also a zero forcing set for~$J_{\{0,1,\dots,t\}}(n',k)$. This means that if we can find a tight construction for the smallest case~$n=2k+1$, this gives a tight bound for all admissible values of~$n$. Such a zero forcing set can be found computationally when~$(n,k,t)\in\{(9,4,1),(11,5,1),(13,6,1)\}$, but not for~$(7,3,1)$ (see~\cite{cocalc} for source code). This suggests that it may be possible to extend Theorem~\ref{th:genKneserUpper} to all triples~$(n,k,t)$ such that~$n \ge 2k+1$ and~$t \le k-3$.

\subsection{Generalized Grassmann graphs}\label{sec:GrassmannZ}

Lemma~\ref{lem:bollobas} has the following analogue for subspaces over a field. 
\begin{lemma}[\cite{furedi84}]
\label{lem:bollobassubspace}
Let~$U_1,U_2,\dots, U_m$ be~$r$-dimensional subspaces and let~$W_1,W_2,\dots,W_m$ be~$s$-dimensional subspaces of a linear space over a field~$\mathbb{F}$. Let $t\ge 0$ such that
\begin{description}
    \item[$(i)$] $\dim(U_i\cap W_i) \le t$ for~$i\in [m]$;
    \item[$(ii)$] $\dim(U_i\cap W_j) \ge t+1$ for~$1\le i<j\le m$.
\end{description}
Then~$m\le \binom{r+s-2t}{r-t}$.
\end{lemma}

Using the above lemma, we obtain a lower bound on~$Z(J_{q,\{0,1,\dots,t\}}(n,k))$, similar to the generalized Johnson case.

\begin{theorem}\label{th:qKneserLower}
Let~$n\ge 2k+1$ and $S=\{0,1,\dots,t\}$ for some $t \in \{0,1,\dots,k-1\}$. Then
$$Z(J_{q,S}(n,k)) \ge \gauss{n}{k}{q} - \binom{2k-2t}{k-t}.$$
\end{theorem}
\begin{proof}
Consider a Z-Grundy dominating sequence~$v_1,v_2,\dots, v_s$ of maximum length in~$J_{q,S}(n,k)$ and pair each~$v_i$ with a vertex~$w_i$ that is footprinted by it. Then~$\dim(v_i\cap w_i)\le t$ and if $j>i$,~$\dim(v_i\cap w_j) \ge t + 1$. Lemma~\ref{lem:bollobassubspace} implies that~$\gamma^Z_{gr}(J_{q,S}(n,k))\le \binom{2k-2t}{k-t}$, hence
    \[Z(J_{q,S}(n,k)) \ge \gauss{n}{k}{q} - \gamma^Z_{gr}(J_{q,S}(n,k))\ge \gauss{n}{k}{q} - \binom{2k-2t}{k-t}.\qedhere\] 
\end{proof}

For a given vector space~$\mathbb{F}^n_q$, let~$\a_1,\a_2,\dots,\a_n$ be an orthogonal basis. Note that the proofs of Theorem~\ref{th:genKneserUpper} and~\ref{th:genKneserEdgeCase} are still valid for~$K_q(n,k)$ if we replace each subset of~$[n]$ with the corresponding set of basis vectors. Therefore, we have the following analogous result for~$J_{q,S}(n,k)$. 

\begin{corollary}
Let~$S\subseteq\{0,1,\dots,k-2\}$ with~$t\coloneqq\max(S)$, and~$n\ge \max(3k-2t+1,2k+1)$. Then
\[Z(J_{q,S}(n,k)) \le Z_t(J_{q,S}(n,k)) \le Z_c(J_{q,S}(n,k)) \le \gauss{n}{k}{q} - \binom{2k-2t}{k-t}.\]
If~$S=\{0,1,\dots,t\}$, equality holds throughout.
\label{th:qKneserUpper}
\end{corollary}

Note that this strategy can not be used to extend Theorem~\ref{th:genJohnson} to generalized Grassmann graphs. As an example, consider the graph~$J_{2,\{1\}}(4,2)$. Theorem~\ref{th:genKneserUpper} implies that the corresponding Johnson graph~$J(4,2)$ has minimum zero forcing set~$|V|\setminus \{\{1,3\},\{1,4\}\}$. In~$J_{2,\{1\}}(4,2)$, the corresponding set $V\setminus\{\{a_1,a_3\},\{a_1,a_4\}\}$ is also zero forcing, but not of minimum cardinality; there exists a significantly smaller zero forcing set
\[V \setminus\{\langle a_1, a_2 \rangle, \langle a_1, a_3\rangle, \langle a_1, a_4\rangle, \langle a_2, a_3\rangle, \langle a_2, a_4\rangle, \langle a_3, a_4\rangle,\langle a_1+a_2, a_3+a_4\rangle\}.\]

\subsection{Hamming graphs}\label{sec:hamminggraphs}

In this section we show the exact zero forcing number of Hamming graphs, extending the results on~$H(n,2)$ and~$H(2,q)$ in~\cite{AIM} and~\cite{A}, respectively. 

We will use the following elementary property, which can be proved using Newton's binomial theorem.

\begin{lemma}
    For any integers~$n\ge 1$,~$q\ge 2$, 
    \[\sum_{\substack{k\in\{0,1,\dots,n\} \\ k\text{ even}}}\binom{n}{2k}(q-1)^{n-2k} = \frac{1}{2}\left(q^n+(q-2)^{n}\right).\]
    \label{lem:evenTerms}
\end{lemma}

\begin{theorem} \label{th:hamupper}
For $n\ge 1,\ q \geq 2$, $Z(H(n,q)) = \frac{1}{2}(q^n+(q-2)^n)$.
\end{theorem}

\begin{proof} Define \(z_{n,q}\coloneqq\frac{1}{2}\left(q^n+(q-2)^n\right)\). First, we show that~$Z(H(n,q)) \le z_{n,q}$.
Define the \textit{core} of~$H(n,q)$ as the set~$C_{n,q} \coloneqq \{ (a_1,a_2,\dots,a_n) \mid a_i\in \{1,2,\dots,q-2\} \}$. We prove by induction on \(n\) that there exists a zero forcing set of size \(z_{n,q}\) that contains \(C_{n,q}\), but where no vertex of \(C_{n,q}\) acts as a pivot. 

 For~$n=1$, the set~$C_{1,q}\cup \{(0)\}$ is a zero forcing set for~$H(1,q)$ of size~$q-1$, and white vertex~$(q-1)$ can be forced with pivot~$(0)\notin C_{1,q}$. 

Suppose that we have a zero forcing set of~$H(n-1,q)$ of size \(z_{n-1,q}\) in which no vertex of~$C_{n-1,q}$ acts as a pivot, but all are in the leader set. As~$H(n,q) = H(n-1,q)\square K_q$, the Hamming graph \(H(n,q)\) can be constructed by taking~$q$ copies of~$H(n-1,q)$ and connecting the corresponding vertices in each copy with an edge. The tuple corresponding to each vertex is then extended with a new entry denoting the copy it is in. In the first~$q-1$ copies, select the same zero forcing set of~$H(n-1,q)$ respecting the above conditions. In the last copy, choose the same set, excluding the core. The first~$q-1$ copies can be forced black by applying the zero forcing process of~$H(n-1,q)$ to each. Note that, by the induction hypothesis, the pivots that we use here are black in every copy, so they still have only one white neighbor. Now we can force the core of the last copy, using the corresponding vertices in the first copy as pivots. Finally, the zero forcing process of~$H(n-1,q)$ can be repeated for the last copy. Note that the core of~$H(n,q)$ is included in the leader set, while none of its vertices were used as a pivot.
\begin{figure}[H]
    \centering
    \vspace{3mm}
\begin{subfigure}[H]{0.28\textwidth}
    \centering
    \vspace{2mm}
    \begin{adjustbox}{width=0.99\textwidth}
    \tdplotsetmaincoords{60}{125}
    \begin{tikzpicture}
		[tdplot_main_coords,
			cube/.style={very thick,black},
			cube2/.style={very thick,dashed},
			cube3/.style={draw=none,fill=inner2}]
			
	\draw[cube3] (0,0,2) -- (0.5,0,2) -- (0.5,2,2) -- (0,2,2) -- cycle;
	\draw[cube3] (0.5,0,1.5) -- (1.5,0,1.5) -- (1.5,1.5,1.5) -- (0.5,1.5,1.5) -- cycle;
	\draw[cube3] (1.5,0,1.5) -- (1.5,1.5,1.5) -- (1.5,1.5,0.5) -- (1.5,2,0.5) -- (1.5,2,0) -- (1.5,0,0) -- cycle;
	\draw[cube3] (0,2,2) -- (0,2,0) -- (1.5,2,0) -- (1.5,2,0.5) -- (0.5,2,0.5) -- (0.5,2,2) -- cycle;
	\draw[cube3] (0.5,0,2) -- (0.5,2,2) -- (0.5,2,0.5) -- (0.5,1.5,0.5) -- (0.5,1.5,1.5) -- (0.5,0,1.5) -- cycle;
	\draw[cube3] (1.5,1.5,0.5) -- (1.5,2,0.5) -- (0.5,2,0.5) -- (0.5,1.5,0.5) --cycle;
	\draw[cube3] (0.5,1.5,1.5) -- (0.5,1.5,0.5) -- (1.5,1.5,0.5) -- (1.5,1.5,1.5) -- cycle;
	
	\draw[cube] (0.5,1.5,2) -- (0.5,0,2) -- (0.5,0,1.5) -- (0.5,1.5,1.5);
	\draw[cube] (0.5,0,2) -- (0,0,2) -- (0,2,2) -- (0.5,2,2) -- (0.5,2,0.5) -- (1.5,2,0.5) -- (1.5,2,0) -- (0,2,0) -- (0,2,2);
	\draw[cube] (0.5,0,1.5) -- (1.5,0,1.5) -- (1.5,1.5,1.5) -- (0.5,1.5,1.5);
	\draw[cube] (1.5,2,0) -- (1.5,0,0) -- (1.5,0,1.5);
    \draw[cube] (0.5,2,2) -- (0.5,1.5,2);
    \draw[cube] (0.5,1.5,1.5) -- (0.5,1.5,0.5) -- (1.5,1.5,0.5) -- (1.5,2,0.5);
    \draw[cube] (1.5,1.5,0.5) -- (1.5,1.5,1.5);
    \draw[cube] (0.5,1.5,0.5) -- (0.5,2,0.5);

	\draw[cube2] (0,0,0) -- (0,2,0) -- (2,2,0) -- (2,0,0) -- cycle;
	\draw[cube2] (0,0,2) -- (0,2,2) -- (2,2,2) -- (2,0,2) -- cycle;
	
	\draw[cube2] (0,0,0) -- (0,0,2);
	\draw[cube2] (0,2,0) -- (0,2,2);
	\draw[cube2] (2,0,0) -- (2,0,2);
	\draw[cube2] (2,2,0) -- (2,2,2);
	
	\node[draw=none,fill=none] at (2.8,0,0){$\small(0,4,0)$};
	\node[draw=none,fill=none] at (0,2.8,0){$\small(0,0,4)$};
	\node[draw=none,fill=none] at (0,0,2.3){$\small(4,0,0)$};
	
	\node[draw=none,fill=none] at (3.7,0,0){};
	\node[draw=none,fill=none] at (0,3.7,0){};
	
\end{tikzpicture}
    \end{adjustbox}
    \vspace{1mm}
\end{subfigure}
\hspace{-8mm}
\begin{subfigure}[H]{0.28\textwidth}
    \centering
    \begin{adjustbox}{width=0.99\textwidth}
    \tdplotsetmaincoords{60}{125}
    \begin{tikzpicture}
		[tdplot_main_coords,
			cube/.style={very thick,black},
			cube2/.style={very thick,dashed},
			cube3/.style={draw=none,fill=inner2},
			cube4/.style={draw=none,fill=black,pattern=north east lines}]
			
	\draw[cube3] (0,0,2) -- (0,2,2) -- (0.5,2,2) -- (0.5,0,2) -- cycle;
	\draw[cube3] (0,2,2) -- (0,2,0) -- (1.5,2,0) -- (1.5,2,0.5) -- (0.5,2,0.5) -- (0.5,2,2) -- cycle;
	\draw[cube3] (0.5,1.5,1.5) -- (1.5,1.5,1.5) -- (1.5,1.5,0.5) -- (0.5,1.5,0.5) -- cycle;
	\draw[cube3] (0.5,1.5,2) -- (0.5,2,2) -- (0.5,2,0.5) -- (0.5,1.5,0.5) -- cycle;
	\draw[cube3] (0.5,1.5,0.5) -- (1.5,1.5,0.5) -- (1.5,2,0.5) -- (0.5,2,0.5) -- cycle;
	\draw[cube3] (1.5,1.5,0) -- (1.5,2,0) -- (1.5,2,0.5) -- (1.5,1.5,0.5) -- cycle;
	\draw[cube4] (0.5,0,2) -- (2,0,2) -- (2,1.5,2) -- (0.5,1.5,2) -- cycle;
	\draw[cube4] (2,0,2) -- (2,0,0) -- (2,1.5,0) -- (2,1.5,2) -- cycle;
	\draw[cube4] (0.5,1.5,2) -- (0.5,1.5,1.5) -- (1.5,1.5,1.5) -- (1.5,1.5,0) -- (2,1.5,0) -- (2,1.5,2) -- cycle;

    \draw[cube] (0.5,0,2) -- (0,0,2) -- (0,2,2) -- (0.5,2,2) -- (0.5,1.5,2);
    \draw[cube] (0,2,2) -- (0,2,0) -- (1.5,2,0) -- (1.5,2,0.5) -- (0.5,2,0.5) -- (0.5,2,2);
    \draw[cube] (0.5,2,0.5) -- (0.5,1.5,0.5) -- (1.5,1.5,0.5) -- (1.5,2,0.5);
    \draw[cube] (0.5,1.5,2) -- (0.5,1.5,1.5) -- (1.5,1.5,1.5) -- (1.5,1.5,0) -- (1.5,2,0);
    \draw[cube] (0.5,1.5,1.5) -- (0.5,1.5,0.5);
    \draw[cube] (2,1.5,0) -- (1.5,1.5,0);
    \draw[cube] (0.5,0,2) -- (2,0,2) -- (2,1.5,2) -- (0.5,1.5,2) -- cycle;
    \draw[cube] (2,0,2) -- (2,0,0) -- (2,1.5,0) -- (2,1.5,2);
	
	\draw[cube2] (0,0,0) -- (0,2,0) -- (2,2,0) -- (2,0,0) -- cycle;
	\draw[cube2] (0,0,2) -- (0,2,2) -- (2,2,2) -- (2,0,2) -- cycle;
	
	\draw[cube2] (0,0,0) -- (0,0,2);
	\draw[cube2] (0,2,0) -- (0,2,2);
	\draw[cube2] (2,0,0) -- (2,0,2);
	\draw[cube2] (2,2,0) -- (2,2,2);
	
	\node[draw=none,fill=none] at (3.7,0,0){};
	\node[draw=none,fill=none] at (0,3.7,0){};
	
\end{tikzpicture}
    \end{adjustbox}
\end{subfigure}
\hspace{-10mm}
\begin{subfigure}[H]{0.28\textwidth}
    \centering
    \begin{adjustbox}{width=0.99\textwidth}
    \tdplotsetmaincoords{60}{125}
    \begin{tikzpicture}
		[tdplot_main_coords,
			cube/.style={very thick,black},
			cube2/.style={very thick,dashed},
			cube3/.style={draw=none,fill=inner2},
			cube4/.style={draw=none,fill=black,pattern=north east lines}]

	\draw[cube3] (0,2,2) -- (0,2,0) -- (1.5,2,0) -- (1.5,2,0.5) -- (0.5,2,0.5) -- (0.5,2,2) -- cycle;
	\draw[cube3] (1.5,1.5,0) -- (1.5,2,0) -- (1.5,2,0.5) -- (1.5,1.5,0.5) -- cycle;
	\draw[cube3] (0,0,2) -- (2,0,2) -- (2,1.5,2) -- (0.5,1.5,2) -- (0.5,2,2) -- (0,2,2) -- cycle;
	\draw[cube3] (2,0,2) -- (2,0,0) -- (2,1.5,0) -- (2,1.5,2) -- cycle;
	\draw[cube3] (0.5,1.5,2) -- (0.5,1.5,1.5) -- (1.5,1.5,1.5) -- (1.5,1.5,0) -- (2,1.5,0) -- (2,1.5,2) -- cycle;
	\draw[cube3] (0.5,1.5,2) -- (0.5,2,2) -- (0.5,2,1.5) -- (0.5,1.5,1.5) -- cycle;
	\draw[cube4] (1.5,1.5,1.5) -- (1.5,2,1.5) -- (1.5,2,0.5) -- (1.5,1.5,0.5) -- cycle;
	\draw[cube4] (1.5,1.5,1.5) -- (0.5,1.5,1.5) -- (0.5,2,1.5) -- (1.5,2,1.5) -- cycle;
	\draw[cube4] (1.5,2,1.5) -- (0.5,2,1.5) -- (0.5,2,0.5) -- (1.5,2,0.5) -- cycle;

    \draw[cube] (0,2,2) -- (0,2,0) -- (1.5,2,0) -- (1.5,2,0.5) -- (0.5,2,0.5) -- (0.5,2,2);
    \draw[cube] (2,0,2) -- (2,0,0) -- (2,1.5,0) -- (2,1.5,2);
    \draw[cube] (0,0,2) -- (2,0,2) -- (2,1.5,2) -- (0.5,1.5,2) -- (0.5,2,2) -- (0,2,2) -- cycle;
    \draw[cube] (2,1.5,0) -- (1.5,1.5,0) -- (1.5,2,0);
    \draw[cube] (1.5,1.5,0) -- (1.5,1.5,1.5) -- (0.5,1.5,1.5) -- (0.5,2,1.5);
    \draw[cube] (1.5,1.5,0.5) -- (1.5,2,0.5) -- (1.5,2,1.5) -- (0.5,2,1.5);
    \draw[cube] (1.5,1.5,1.5) -- (1.5,2,1.5);
    \draw[cube] (0.5,1.5,2) -- (0.5,1.5,1.5);
	
	\draw[cube2] (0,0,0) -- (0,2,0) -- (2,2,0) -- (2,0,0) -- cycle;
	\draw[cube2] (0,0,2) -- (0,2,2) -- (2,2,2) -- (2,0,2) -- cycle;
	
	\draw[cube2] (0,0,0) -- (0,0,2);
	\draw[cube2] (0,2,0) -- (0,2,2);
	\draw[cube2] (2,0,0) -- (2,0,2);
	\draw[cube2] (2,2,0) -- (2,2,2);
	
	\node[draw=none,fill=none] at (3.7,0,0){};
	\node[draw=none,fill=none] at (0,3.7,0){};
	
\end{tikzpicture}
    \end{adjustbox}
\end{subfigure}
\hspace{-10mm}
\begin{subfigure}[H]{0.28\textwidth}
    \centering
\begin{adjustbox}{width=0.99\textwidth}
    \tdplotsetmaincoords{60}{125}
    \begin{tikzpicture}
		[tdplot_main_coords,
			cube/.style={very thick,black},
			cube2/.style={very thick,dashed},
			cube3/.style={draw=none,fill=inner2},
			cube4/.style={draw=none,fill=black,pattern=north east lines}]

	\draw[cube3] (0,0,2) -- (2,0,2) -- (2,1.5,2) -- (0.5,1.5,2) -- (0.5,2,2) -- (0,2,2) -- cycle;
	\draw[cube3] (2,0,2) -- (2,0,0) -- (2,1.5,0) -- (2,1.5,2) -- cycle;
	\draw[cube3] (0,2,2) -- (0,2,0) -- (1.5,2,0) -- (1.5,2,1.5) -- (0.5,2,1.5) -- (0.5,2,2) -- cycle;
	\draw[cube4] (2,1.5,2) -- (2,2,2) -- (2,2,0) -- (2,1.5,0) -- cycle;
	\draw[cube4] (2,2,2) -- (0.5,2,2) -- (0.5,1.5,2) -- (2,1.5,2) -- cycle;
	\draw[cube4] (0.5,2,2) -- (0.5,2,1.5) -- (1.5,2,1.5) -- (1.5,2,0) -- (2,2,0) -- (2,2,2) -- cycle;
	
	\draw[cube] (0,0,2) -- (0,2,2) -- (2,2,2) -- (2,0,2) --  cycle;
	\draw[cube] (0,2,2) -- (0,2,0) -- (2,2,0) -- (2,0,0) -- (2,0,2);
	\draw[cube] (2,2,0) -- (2,2,2);
	\draw[cube] (0.5,1.5,2) -- (0.5,2,2) -- (0.5,2,1.5) -- (1.5,2,1.5) -- (1.5,2,0);
	\draw[cube] (0.5,1.5,2) -- (2,1.5,2) -- (2,1.5,0);

	\draw[cube2] (0,0,0) -- (0,2,0) -- (2,2,0) -- (2,0,0) -- cycle;
	\draw[cube2] (0,0,2) -- (0,2,2) -- (2,2,2) -- (2,0,2) -- cycle;
	
	\draw[cube2] (0,0,0) -- (0,0,2);
	\draw[cube2] (0,2,0) -- (0,2,2);
	\draw[cube2] (2,0,0) -- (2,0,2);
	\draw[cube2] (2,2,0) -- (2,2,2);
	
	\node[draw=none,fill=none] at (3.7,0,0){};
	\node[draw=none,fill=none] at (0,3.7,0){};
	
\end{tikzpicture}
    \end{adjustbox}
\end{subfigure}
\hspace{-10mm}
    \caption{The zero forcing process for~$H(3,4)$. The Hamming graph is depicted as a cube, with each unit subcube representing a vertex. Vertices are adjacent whenever their cubes line up along one of the three main axes.}
    \label{fig:HammingGraphUpper}
\end{figure}
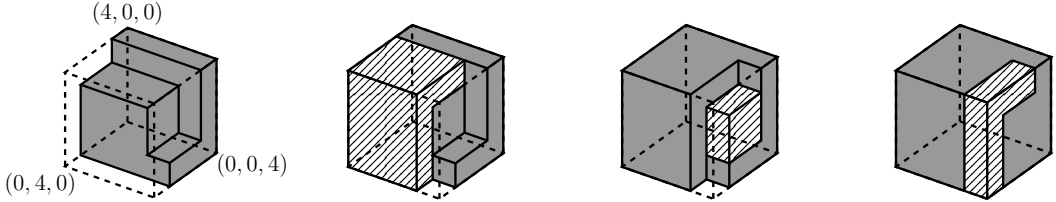
The constructed zero forcing set has size \[qz_{n-1,q}-(q-2)^{n-1}=z_{n,q}\] and therefore completes the induction. We conclude that \(Z(H(n,q))\leq z_{n,q}\).

To show that this bound is tight, we now prove that~$Z(H(n,q))\ge z_{n,q}$. Define the matrix \(B_n\) recursively as
    \[\left\{\begin{matrix*}[l]B_1=J\\B_n=J\otimes I^{\otimes(n-1)}+I\otimes B_{n-1}\end{matrix*}\right.\]
    or, explicitly, \[B_n=\underbrace{J\otimes I\otimes\cdots\otimes I}_n+\underbrace{I\otimes J\otimes \cdots\otimes I}_n+\cdots+\underbrace{I\otimes\cdots\otimes I\otimes J}_n\] where \(I\) and \(J\) are the \(q\times q\) unit matrix and the \(q\times q\) all-one matrix respectively, and~$\otimes$ denotes the tensor product. Then \(B_n\) is in~$\Sa(H(n,q))$.
    We show that its nullity as a matrix over \(\mathbb{F}_2\) is at least \(z_{n,q}\) by constructing a set of \(z_{n,q}\) independent eigenvectors with eigenvalue zero. The result then follows from Lemma~\ref{lem:maxnul}. Let \(\e_i\) be the \(i\)th standard basis vector of \(\mathbb{F}_2^q\) and let \(\j\) denote the all-one vector of \(\mathbb{F}_2^q\). This vector is an eigenvector of \(J\) with eigenvalue one if \(q\) is odd and eigenvalue zero if \(q\) is even. Let \(\{\x_1,\x_2,\dots,\x_{q-1}\}\) be a set of independent eigenvectors of \(J\) with eigenvalue zero.

    \begin{description}
    \item[Case 1: \(q\) is odd.] Consider all elements of \(\{\j,\x_1,\x_2,\dots,\x_{q-1}\}^{\otimes n}\) that contain an even number of \(\j\)'s. In other words, all tensors of the form \[\underbrace{\j\otimes\j\otimes\cdots\otimes\j}_k\;\otimes\;\x_{i_1}\otimes \x_{i_2}\otimes\cdots\otimes \x_{i_{n-k}}\] with \(k\) even, and permutations thereof. These are all elements of the kernel of \(B_n\) and, by Lemma~\ref{lem:evenTerms}, there are \(z_{n,q}\) such vectors. Moreover, they are all linearly independent, because \(\x_1,\x_2,\dots,\x_{q-1}\) and \(\j\) are linearly independent.

    \item[Case 2: \(q\) is even.]
    
    We construct \(z_{n,q}\) eigenvectors with eigenvalue zero by induction on \(n\). If \(n=1\), then we can choose the \(z_{1,q}=q-1\) eigenvectors \(\x_1,\x_2,\dots,\x_{q-1}\) of \(J\). For the induction step, suppose that \(X\) is a set of \(z_{n-1,q}\) linearly independent vectors that nullify \(B_{n-1}\).

    Since \(J\j=0\), we can choose the set \(\{\x_1,\x_2,\dots,\x_{q-1}\}\) in such a way that \(\x_1=\j\). Consider the \((q-2)z_{n-1,q}\) vectors of the form
    \[\x_{i}\otimes\mathbf{v}\]
    with \(i\in\{2,3,\dots,q-1\}\) and \(\mathbf{v}\in X\). They are eigenvectors with eigenvalue zero.

    Consider also the \(q^{n-1}\) vectors of the form
    \[\j\otimes\w+\e_1\otimes(B_{n-1}\w)\]
    with \(\w\in\{\e_1,\e_2,\dots,\e_q\}^{\otimes(n-1)}\). They are nullified by \(B_n\) since
    {\small{
    \begin{align*}
        B_n(\j\otimes\w+\e_i\otimes(B_{n-1}\w))&=(J\otimes I+I\otimes B_{n-1})(\j\otimes\w+\e_i\otimes(B_{n-1}\w))\\
        &=2\cdot\j\otimes(B_{n-1}\w)+\e_i\otimes(B_{n-1}^2\w)\\
        &=\mathbf{0},
    \end{align*}
    }}
    
    where we used that \(B_{n-1}^2=0\) since \(J^2=qJ=0\).
    Moreover, the vectors \(\x_{i}\otimes\mathbf{v}\) and \(\j\otimes\w+\e_1\otimes(B_{n-1}\w)\) are all linearly independent because the \(\x_i\), \(\j\) and \(\e_1\) are linearly independent, and all \(\mathbf{v}\) are linearly independent by the induction hypothesis. In total, we obtain
    \[(q-2)z_{n-1,q}+q^{n-1}=\frac{1}{2}(q^n+(q-2)^n)=z_{n,q}\]
    linearly independent eigenvectors with eigenvalue zero.
 \qedhere   \end{description}
\end{proof}

\subsection*{Acknowledgements} 
Aida Abiad is supported by the Dutch Research Council through the grant VI.Vidi.213.085. Robin Simoens is supported by the Research Foundation Flanders through the grant 11PG724N. The authors thank  Jozefien D'haeseleer, Cor Hurkens and Nick Reniers for inspiring discussions in the early stage of this work.

\bibliographystyle{abbrv}

\begin{thebibliography}{10}

\bibitem{aazami2008hardness}
A.~Aazami.
\newblock {\em Hardness results and approximation algorithms for some problems
  on graphs}.
\newblock PhD thesis, University of Waterloo, 2008.

\bibitem{A2010}
A.~Aazami.
\newblock Domination in graphs with bounded propagation: algorithms,
  formulations and hardness results.
\newblock {\em Journal of Combinatorial Optimization}, 19:429--456, 2010.

\bibitem{agong2018girth}
L.~A. Agong, C.~Amarra, J.~S. Caughman, A.~J. Herman, and T.~S. Terada.
\newblock On the girth and diameter of generalized {J}ohnson graphs.
\newblock {\em Discrete Mathematics}, 341(1):138--142, 2018.

\bibitem{AIM}
{AIM Minimum Rank – Special Graphs Work Group (F.~Barioli, W.~Barrett, S.~Butler, S.~M.~Cioaba, D.~Cvetkovi\'c, S.~M.~Fallat, C.~D.~Godsil, W.~H.~Haemers, L.~Hogben,
R.~Mikkelson, S.~Narayan, O.~Pryporova, I.~Sciriha, W.~So, D.~Stevanovi\'c, H.~{\VAN{Holst}{Van der}{van der}}~Holst, K.~Vander Meulen, and A.~Wangsness Wehe)}.
\newblock Zero forcing sets and the minimum rank of graphs.
\newblock {\em Linear Algebra and its Applications}, 428:1628–1648, 2008.

\bibitem{A}
N.~Alon.
\newblock A propagation process on {C}ayley graphs.
\newblock 2008.
\newblock \url{ https://www.cs.tau.ac.il//~nogaa/PDFS/pn.pdf}

\bibitem{alspach2013johnson}
B.~Alspach.
\newblock {J}ohnson graphs are {H}amilton-connected.
\newblock {\em Ars Mathematica Contemporanea}, 6(1):21--23, 2013.

\bibitem{balogh2019coloring}
J.~Balogh, D.~Cherkashin, and S.~Kiselev.
\newblock Coloring general {K}neser graphs and hypergraphs via high-discrepancy
  hypergraphs.
\newblock {\em European Journal of Combinatorics}, 79:228--236, 2019.

\bibitem{barioli2013parameters}
F.~Barioli, W.~Barrett, S.~M. Fallat, H.~T. Hall, L.~Hogben, B.~Shader,
  P.~{\VAN{Driessche}{Van den}{van den}}~Driessche, and H.~{\VAN{Holst}{Van der
  }{van der}}~Holst.
\newblock Parameters related to tree-width, zero forcing, and maximum nullity
  of a graph.
\newblock {\em Journal of Graph Theory}, 72(2):146--177, 2013.

\bibitem{bose1966characterization}
R.~C. Bose and R.~Burton.
\newblock A characterization of flat spaces in a finite geometry and the
  uniqueness of the {H}amming and the {M}ac{D}onald codes.
\newblock {\em Journal of Combinatorial Theory}, 1(1):96--104, 1966.

\bibitem{brevsar2016dominating}
B.~Bre{\u{s}}ar, T.~Gologranc, and T.~Kos.
\newblock Dominating sequences under atomic changes with applications in
  {S}ierpinski and interval graphs.
\newblock {\em Applicable Analysis and Discrete Mathematics}, 10(2):518--531,
  2016.

\bibitem{brevsar2019grundy}
B.~Bre{\u{s}}ar, T.~Kos, and P.~D. Torres.
\newblock Grundy domination and zero forcing in {K}neser graphs.
\newblock {\em Ars Mathematica Contemporanea}, 17(2):419--430, 2019.

\bibitem{BRESAR2017}
B.~Bre\u{s}ar, C.~Bujt\'as, T.~Gologranc, S.~Klav\u{z}ar, G.~Ko\u{s}mrlj,
  B.~Patk\'os, Z.~Tuza, and M.~Vizer.
\newblock Grundy dominating sequences and zero forcing sets.
\newblock {\em Discrete Optimization}, 26:66--77, 2017.

\bibitem{BRIMKOV2019}
B.~Brimkov, C.~C. Fast, and I.~V. Hicks.
\newblock Computational approaches for zero forcing and related problems.
\newblock {\em Discrete Optimization}, 273:889--903, 2019.

\bibitem{BRIMKOV201731}
B.~Brimkov and I.~V. Hicks.
\newblock Complexity and computation of connected zero forcing.
\newblock {\em Discrete Applied Mathematics}, 229:31--45, 2017.


\bibitem{burgarth2009local}
D.~Burgarth, S.~Bose, C.~Bruder and V.~Giovannetti.
\newblock Local controllability of quantum networks.
\newblock {\em Physical Review A}, 79(6):060305, 2009.

\bibitem{CHT2023} J.~S. Caughman, A.~J. Herman, T.~S. Terada. The girth, odd girth, distance function, and diameter of generalized Johnson graphs. arXiv:2304.02864.

\bibitem{CIOABA2018219}
S.~M. Cioabă, W.~H. Haemers, T.~Johnston, and M.~McGinnis.
\newblock Cospectral mates for the union of some classes in the Johnson association scheme.
\newblock {\em Linear Algebra and its Applications}, 539:219--228, 2018.

\bibitem{totalZF}
R.~Davila and M.~A. Henning.
\newblock On the total forcing number of a graph.
\newblock {\em Discrete Applied Mathematics}, 257:115--127, 2019.

\bibitem{dean2011power}
N.~Dean, A.~Ilic, I.~Ramirez, J.~Shen, and K.~Tian.
\newblock On the power dominating sets of hypercubes.
\newblock {\em 2011 14th IEEE International Conference on Computational Science and Engineering}, 488--491, 2011.

\bibitem{eroh2017comparison}
L.~Eroh, C.~X. Kang, and E.~Yi.
\newblock A comparison between the metric dimension and zero forcing number of
  trees and unicyclic graphs.
\newblock {\em Acta Mathematica Sinica, English Series}, 33(6):731--747, 2017.

\bibitem{ZplusJohnson}
S.~Fallat, K.~Meagher, A.~Soltani, and B.~Yang.
\newblock Compressed cliques graphs, clique coverings and positive zero
  forcing.
\newblock {\em Theoretical Computer Science}, 734:119--130, 2018.

\bibitem{fallat2016complexity}
S.~Fallat, K.~Meagher, and B.~Yang.
\newblock On the complexity of the positive semidefinite zero forcing number.
\newblock {\em Linear Algebra and its Applications}, 491:101--122, 2016.

\bibitem{furedi84}
Z.~Füredi.
\newblock Geometrical solution of an intersection problem for two hypergraphs.
\newblock {\em European Journal of Combinatorics}, 5(2):133--136, 1984.

\bibitem{IEPbook}
L.~Hogben, J.-H. Lin, and B.~Shader.
\newblock {\em Inverse Problems and Zero Forcing for Graphs}.
\newblock American Mathematical Society in the Mathematical Surveys and Monographs series, 270, 2022.

\bibitem{meagher2012metric}
K.~Meagher and R.~F. Bailey.
\newblock On the metric dimension of {G}rassmann graphs.
\newblock {\em Discrete Mathematics \& Theoretical Computer Science}, 13, 2012.

\bibitem{TD2015}
M.~Trefoy and J.-C. Delvenne.
\newblock Zero forcing number, constrained matchings and strong structural
  controllability.
\newblock {\em Linear Algebra and its Applications}, 484:199--218, 2015.

\bibitem{valencia2005diameter}
M.~Valencia-Pabon and J.-C. Vera.
\newblock On the diameter of {K}neser graphs.
\newblock {\em Discrete Mathematics}, 305(1-3):383--385, 2005.

\bibitem{YANG2013100}
B.~Yang.
\newblock Fast–mixed searching and related problems on graphs.
\newblock {\em Theoretical Computer Science}, 507:100--113, 2013.

\bibitem{cocalc}
S.~Zeijlemaker.
\newblock {Sage code for the zero forcing number of graph classes in the Johnson, Grassmann and Hamming schemes}.
\newblock \url{https://cocalc.com/share/public_paths/9860d0b7b4e62e87e4389113d30f0ac9538f129d/Zero_forcing_generalized_Grassmann_Johnson_Hamming_graphs.ipynb}

\end{thebibliography}

\DeclareRobustCommand{\VAN}[3]{#3}

\end{document}